\documentclass[12pt,a4paper,draft]{amsart}

\usepackage{hyperref} 
\usepackage{amsrefs}
\usepackage{amssymb}

\numberwithin{equation}{section}

\let\Im=\undefined
\let\mod=\undefined
\DeclareMathOperator{\GL}{GL}
\DeclareMathOperator{\Id}{Id}
\DeclareMathOperator{\Im}{Im}
\DeclareMathOperator{\End}{End}
\DeclareMathOperator{\ext}{ext}
\DeclareMathOperator{\Ext}{Ext}
\DeclareMathOperator{\Hom}{Hom}
\DeclareMathOperator{\mod}{mod}
\DeclareMathOperator{\idim}{idim}
\DeclareMathOperator{\pdim}{pdim}
\DeclareMathOperator{\rank}{rank}
\DeclareMathOperator{\gldim}{gldim}
\DeclareMathOperator{\bdim}{\mathbf{dim}}

\newcommand{\bd}{\mathbf{d}}
\newcommand{\bh}{\mathbf{h}}
\newcommand{\bone}{\mathbf{1}}

\newcommand{\bbA}{\mathbb{A}}
\newcommand{\bbB}{\mathbb{B}}
\newcommand{\bbM}{\mathbb{M}}
\newcommand{\bbN}{\mathbb{N}}
\newcommand{\bbP}{\mathbb{P}}
\newcommand{\bbX}{\mathbb{X}}
\newcommand{\bbV}{\mathbb{V}}
\newcommand{\bbZ}{\mathbb{Z}}

\newcommand{\calC}{\mathcal{C}}
\newcommand{\calE}{\mathcal{E}}

\newcommand{\calL}{\mathcal{L}}
\newcommand{\calO}{\mathcal{O}}
\newcommand{\calP}{\mathcal{P}}
\newcommand{\calR}{\mathcal{R}}
\newcommand{\calT}{\mathcal{T}}
\newcommand{\calU}{\mathcal{U}}
\newcommand{\calV}{\mathcal{V}}
\newcommand{\calX}{\mathcal{X}}

\newcommand{\calZ}{\mathcal{Z}}

\newcommand{\frakR}{\mathfrak{R}}

\newcommand{\ol}{\overline}

\newtheorem{maintheorem}{Theorem}

\newcounter{claim}
\numberwithin{claim}{section}
\newtheorem{corollary}[claim]{Corollary}
\newtheorem{lemma}[claim]{Lemma}
\newtheorem{proposition}[claim]{Proposition}
\newtheorem{theorem}[claim]{Theorem}

\newtheorem{step}{Step}

\title[The variety of periodic modules]{\makebox[0pt]{The closure of the set of periodic modules} \\ over a concealed canonical algebra \\ is regular in codimension one}

\author{Grzegorz Bobi\'nski}

\author{Grzegorz Zwara}

\address{Faculty of Mathematics and Computer Science \\ Nicolaus Copernicus University \\ ul.~Chopina 12/18 \\ 87-100 Toru\'n \\ Poland}

\email{gregbob@mat.umk.pl}

\email{gzwara@mat.umk.pl}

\keywords{concealed canonical algebra, module variety, regularity in codimension one, periodic module}

\subjclass[2010]{16G20, 14L30, 14B05}

\begin{document}

\begin{abstract}
Let $\Lambda$ be a concealed canonical algebra and $\bd$ the dimension vector of a $\Lambda$-module which is periodic respect to the action of the Auslander--Reiten translation $\tau$. In the paper, we investigate the union of the closures of the orbits of the $\tau$-periodic $\Lambda$-modules of dimension
vector $\bd$. We show that this set is closed and regular in codimension one.
\end{abstract}

\maketitle

\section*{Introduction and the main result}

Throughout the paper $k$ is a fixed algebraically closed field. By $\bbZ$, $\bbN$, and $\bbN_+$ we denote the sets of integers, nonnegative integers, and positive integers, respectively. If $i, j \in \bbZ$, then $[i, j]$ denotes the set of all $l \in \bbZ$ such that $i \leq l \leq j$.

For a finite dimensional $k$-algebra $\Lambda$ and a nonnegative element $\bd$ of the Grothendieck group, one defines a variety $\mod_\Lambda^\bd (k)$ (see Section~\ref{sect geometric} for details), whose points parameterize the $\Lambda$-modules with dimension vector $\bd$ (i.e.\ modules whose Jordan--H\"older composition factors are given by $\bd$). If $M$ is $\Lambda$-module of dimension vector $\bd$, then we denote by $\calO_M$ the set of points of $\mod_\Lambda^\bd (k)$, which correspond to modules isomorphic to $M$.

The variety $\mod_\Lambda^\bd (k)$ is in general reducible. However, if $\Lambda$ is triangular (there is no sequence $P_0 \to P_1 \to \cdots \to P_n$ of nonzero non\-isomorphisms between indecomposable projective $\Lambda$-modules such that $n > 0$ and $P_0 \simeq P_n$), then the points corresponding to modules with projective dimension at most $1$ form an open and irreducible (if nonempty) subset $\calP (\bd)$ of $\mod_\Lambda^\bd (k)$ (\cite{BarotSchroer}*{Proposition~3.1}). Consequently, if $\calP (\bd) \neq \emptyset$, then $\ol{\calP (\bd)}$ is an irreducible component of $\mod_\Lambda (\bd)$.

Our ongoing aim is to study components of the form $\ol{\calP (\bd)}$. We expect they are normal (\cite{BobinskiZwara} provides an evidence for this belief) or at least regular in codimension one. There are two sources of such components. First, if $M$ is a directing (see Subsection~\ref{sub directing} for the definition) $\Lambda$-module of dimension vector $\bd$, then $\ol{\calO_M} = \ol{\calP (\bd)}$. The components of this form were studied in~\cite{Bobinski2009} and the main result of~\cite{Bobinski2009} states they are regular in codimension one.

We describe now another class of such components. An important role in the representation theory of algebras is played by the Auslander--Reiten translation $\tau = \tau_\Lambda$, which is an operation on the $\Lambda$-modules (see for example~\cite{AssemSimsonSkowronski}*{Chapter~IV} for the definition). A $\Lambda$-module $M$ is called $\tau$-periodic, if $M \simeq \tau^n M$ for some $n \in \bbN_+$. The algebra $\Lambda$ is called concealed canonical if there exists a tilting line bundle $T$ over a weighted projective line such that $\Lambda$ is isomorphic to the opposite algebra of $\End (T)$. If $\Lambda$ is a concealed canonical algebra, then every $\tau$-periodic $\Lambda$-module has projective dimension at most $1$. Moreover, if $\calU$ is the set of points of $\mod_\Lambda^\bd (k)$ corresponding to the $\tau$-periodic modules and $\calU \neq \emptyset$, then $\ol{\calU} = \ol{\calP (\bd)}$. Geometry of this set and related problems (including study of semiinvariants) are objects of interest, especially in the case of Ringel's canonical algebras~\cite{Ringel1984} -- see~\cites{BarotSchroer, BobinskiSkowronski, Bobinski2008, Bobinski2015, DomokosLenzing2000, DomokosLenzing2002, SkowronskiWeyman} for some results. The following theorem is the main result of the paper.

\begin{maintheorem} \label{main theorem}
Let $\Lambda$ be a concealed canonical algebra and $\bd$ a dimension vector. If
\[
\calU := \{ M \in \mod_\Lambda^\bd (k) : \text{$M$ is $\tau$-periodic} \},
\]
then $\ol{\calU}$ is regular in codimension one.
\end{maintheorem}

According to famous Drozd's Wild--Tame Dichotomy Thoerem~\cite{Drozd} (see also~\cite{CrawleyBoevey}) the algebras can be divided into two disjoint classes. One class consists of the tame algebras for which the indecomposable modules occur, in each dimension, in a finite number of discrete and a finite number of one-parameter families. The second class is formed by the wild algebras whose representation theory is as complicated as the study of finite dimensional vector spaces together with two (noncommuting) endomorphisms, for which the classification up to isomorphism is a well-known unsolved problem.

If $\Lambda$ is a tame concealed canonical algebra and $\bd$ a dimension vector of a $\tau$-periodic $\Lambda$-module, then Theorem~\ref{main theorem} is a consequence of~\cite{BobinskiSkowronski}*{Theorem~1}. More precisely, if this is the case, then $\ol{\calU} = \mod_\Lambda^\bd (k)$ and $\mod_\Lambda^\bd (k)$ is normal. Moreover, results of~\cite{BobinskiSkowronski}*{Section~3} imply, that there exists a convex subalgebra $\Lambda'$ of $\Lambda$, which is also concealed canonical, such that the support algebra of $\bd$ is contained in $\Lambda'$ and
\[
\mod_\Lambda^\bd (k) = \bigcup_{M \in \calU'} \ol{\calO}_M,
\]
where
\[
\calU' := \{ M \in \mod_\Lambda^\bd (k) : \text{$M$ is $\tau_{\Lambda'}$-periodic} \}.
\]
Note that in general $\tau_{\Lambda'}$ and $\tau_\Lambda$ differ, hence $\calU' \neq \calU$.

In the course of the proof of Theorem~\ref{main theorem} we obtain the following analogue of the above result in the case of wild concealed canonical algebras, which seems to be of interest on its own.

\begin{maintheorem} \label{main theorem prim}
Let $\Lambda$ be a wild concealed canonical algebra and $\bd$ a dimension vector. If
\[
\calU := \{ M \in \mod_\Lambda^\bd (k) : \text{$M$ is $\tau$-periodic} \},
\]
then
\[
\ol{\calU} = \bigcup_{M \in \calU} \ol{\calO}_M.
\]
\end{maintheorem}

The paper is organized as follows. In Section~\ref{sect quivers} we present preliminaries on quivers and their representations. Next, in Section~\ref{sect geometric} we define geometric objects of interest and their variants, which play a role in the proof. Sections~\ref{sect nonsingularity} and~\ref{sect sequence} are crucial: the former contains a nonsingularity criterion which we use in the proof, while the latter shows existence of an exact sequence, which is necessary in order to apply the nonsingularity criterion. Section~\ref{section main result} contains more general versions of  Theorems~\ref{main theorem} and~\ref{main theorem prim} and in final Section~\ref{sect applications} we explain how these versions imply Theorems~\ref{main theorem} and~\ref{main theorem prim}. Moreover, in Section~\ref{sect applications} we correct some arguments from~\cite{Bobinski2009} and~\cite{Bobinski2012}.

The authors acknowledge the support of National Science Center grant no.\ 2015/17/B/ST1/01731.

\section{Preliminaries on quivers and their representations} \label{sect quivers}

By a quiver $\Delta$ we mean a finite set $\Delta_0$ of vertices and a finite set $\Delta_1$ of arrows together with two maps $s, t \colon \Delta_1 \to \Delta_0$, which assign to $\alpha \in \Delta_1$ the starting vertex $s \alpha$ and the terminating vertex $t \alpha$, respectively. If $l \in \bbN_+$, then by a path in $\Delta$ of length $l$ we mean every sequence $\sigma = \alpha_1 \cdots \alpha_l$ such that $\alpha_i \in \Delta_1$, for each $i \in [1, l]$, and $s \alpha_i = t \alpha_{i + 1}$, for each $i \in [1, l - 1]$. In the above situation we put $s \sigma := s \alpha_l$ and $t \sigma := t \alpha_1$. Moreover, for each $x \in \Delta_0$, we introduce the path $\bone_x$ in $\Delta$ of length $0$ such that $s \bone_x := x =: t \bone_x$. We denote the length of a path $\sigma$ in $\Delta$ by $\ell (\sigma)$. If $\sigma'$ and $\sigma''$ are two paths in $\Delta$ such that $s \sigma' = t \sigma''$, then we define the composition $\sigma' \sigma''$ of $\sigma'$ and $\sigma''$, which is a path in $\Delta$ of length $\ell (\sigma') + \ell (\sigma'')$, in the obvious way (in particular, $\sigma \bone_{s \sigma} = \sigma = \bone_{t \sigma} \sigma$, for each path $\sigma$).

With a quiver $\Delta$ we associate its path algebra $k \Delta$, which as a $k$-vector space has a basis formed by all paths in $\Delta$ and whose multiplication is induced by the composition of paths. If $\rho \in \bone_x (k \Delta) \bone_y$, for $x, y \in \Delta_0$, then we put $s_\rho := y$ and $t_\rho := x$. If additionally $\rho = \lambda_1 \sigma_1 + \cdots + \lambda_n \sigma_n$, for $\lambda_1, \ldots, \lambda_n \in k$ and paths $\sigma_1$, \ldots, $\sigma_n$ such that $\ell (\sigma_i) > 1$, for each $i \in [1, n]$, then we call $\rho$ a relation. A set $\frakR$ of relations is called minimal if, for every $\rho \in \frakR$, $\rho$ does not belong to the ideal $\langle \frakR \setminus \{ \rho \} \rangle$ of $k \Delta$ generated by $\frakR \setminus \{ \rho \}$. A pair $(\Delta, \frakR)$ consisting of a quiver $\Delta$ and a minimal set of relations $\frakR$, such that there exists $n \in \bbN_+$ with the property $\sigma \in \langle \frakR \rangle$, for each path $\sigma$ in $\Delta$ of length at least $n$, is called a bound quiver. If $(\Delta, \frakR)$ is a bound quiver, then the algebra $k \Delta / \langle \frakR \rangle$ is called the path algebra of $(\Delta, \frakR)$.

For the rest of the section we fix the path algebra $\Lambda$ of a bound quiver $(\Delta, \frakR)$.

Let $R$ be a commutative $k$-algebra. An $R$-representation $M$ of $\Delta$ associates with each vertex $x \in \Delta_0$ a free $R$-module $M_x$ of finite rank and with each arrow $\alpha \in \Delta_1$ an $R$-linear map $M_\alpha : M_{s \alpha} \to M_{t \alpha}$. If $M$ is an $R$-representation of $\Delta$ and $\sigma = \alpha_1 \cdots \alpha_n$ is a path in $\Delta$ with $\alpha_1, \ldots, \alpha_n \in \Delta_1$, then we put
\[
M_\sigma := M_{\alpha_1} \cdots M_{\alpha_n}.
\]
Similarly, if $\rho = \lambda_1 \sigma_1 + \cdots + \lambda_n \sigma_n$ is a relation, for $\lambda_1, \ldots, \lambda_n \in k$ and paths $\sigma_1$, \ldots, $\sigma_n$, then
\[
M_\rho := \lambda_1 M_{\sigma_1} + \cdots + \lambda_n M_{\sigma_n}.
\]
An $R$-representation $M$ of $\Delta$ is called an $R$-representation of $(\Delta, \frakR)$ if $M_\rho = 0$, for each $\rho \in \frakR$. By a morphism $f : M \to N$ of $R$-representations we mean a collection $(f_x)_{x \in  \Delta_0}$ of $R$-linear maps $f_x : M_x \to N_x$, $x \in \Delta_0$, such that $f_{t \alpha} M_\alpha = N_\alpha f_{s \alpha}$, for each $\alpha \in \Delta_1$. The category of $R$-representations of $(\Delta, \frakR)$ is equivalent to the full subcategory $\mod_\Lambda (R)$ of the category of $\Lambda$-$R$-bimodules formed by the bimodules $M$ such that $M_x := \bone_x  M$ is a free $R$-module, for each $x \in \Delta_0$ (see for example~\cite{AssemSimsonSkowronski}*{Theorem~III.1.6} for this statement in the case $R = k$). We will identify such $\Lambda$-$R$-bimodules and the $R$-representations of $(\Delta, \frakR)$.

We write $\mod \Lambda$ for $\mod_\Lambda (k)$. For $M, N \in \mod \Lambda$, we denote by $[M, N]_\Lambda$ the $k$-dimension of $\Hom_\Lambda (M, N)$. Similarly, for $n \in \bbN_+$, we denote by $[M, N]_\Lambda^n$ the $k$-dimension of $\Ext_\Lambda^n (M, N)$. For $M \in \mod \Lambda$, we denote by $\bdim M$ the dimension vector of $M$, which is an element of $\bbN^{\Delta_0}$ such that, for $x \in \Delta_0$, $(\bdim M)_x$ is the $k$-dimension of $M_x$.

If $\gldim \Lambda < \infty$, then one defines the bilinear form $\langle -, - \rangle_\Lambda \colon \bbZ^{\Delta_0} \times \bbZ^{\Delta_0} \to \bbZ$ by the condition
\[
\langle \bdim M, \bdim N \rangle_\Lambda = \sum_{n \in \bbN} (-1)^n [M, N]_\Lambda^n,
\]
for all $\Lambda$-modules $M$ and $N$. We denote the associated quadratic form by $\chi_\Lambda$, i.e.\ $\chi_\Lambda (\bd) := \langle \bd, \bd \rangle_\Lambda$, for $\bd \in \bbZ^{\Delta_0}$. If additionally $\Lambda$ is triangular and $\gldim \Lambda \leq 2$, then
\[
\langle \bd_1, \bd_2 \rangle_\Lambda = \sum_{x \in \Delta_0} d_1 (x) d_2 (x) - \sum_{\alpha \in \Delta_1} d_1 (s \alpha) d_2 (t \alpha) + \sum_{\rho \in \frakR} d_1 (s \rho) d_2 (t \rho),
\]
for all $\bd_1, \bd_2 \in \bbZ^{\Delta_0}$ (see~\cite{Bongartz1983}*{Section~1}).

\section{Geometric preliminaries} \label{sect geometric}

Let $(\Delta, \frakR)$ be a bound quiver and $\Lambda$ its path algebra. We define geometric objects associated with $(\Delta, \frakR)$ and $\Lambda$.

If $\bd_1$ and $\bd_2$ are dimension vectors, then we define affine schemes $\bbV^{\bd_1, \bd_2}$ and $\bbA^{\bd_1, \bd_2}$ by
\[
\bbV^{\bd_1, \bd_2} := \prod_{x \in \Delta_0} \bbM_{d_2 (x), d_1 (x)}
\qquad
\text{and} \qquad
\bbA^{\bd_1, \bd_2} := \prod_{\alpha \in \Delta_1} \bbM_{d_2 (t \alpha), d_1 (s \alpha)},
\]
i.e.\ if $R$ is a commutative $k$-algebra, then
\begin{gather*}
\bbV^{\bd_1, \bd_2} (R) := \prod_{x \in \Delta_0} \bbM_{d_2 (x), d_1 (x)} (R)
\\
\intertext{and}
\bbA^{\bd_1, \bd_2} (R) := \prod_{\alpha \in \Delta_1} \bbM_{d_2 (t \alpha), d_1 (s \alpha)} (R).
\end{gather*}
If $R$ is a commutative $k$-algebra, $M \in \bbA^{\bd_1, \bd_2} (R)$ and $h \in \bbV^{\bd_2, \bd_3} (R)$, then we define $h \circ M \in \bbA^{\bd_1, \bd_3} (R)$ by
\[
(h \circ M)_\alpha := h_{t \alpha} M_\alpha \qquad (\alpha \in \Delta_1).
\]
Analogously, we define $M \circ h \in \bbA^{\bd_1, \bd_3} (R)$, for $h \in \bbV^{\bd_1, \bd_2} (R)$ and $M \in \bbA^{\bd_2, \bd_3} (R)$.

If $\bd$ is a dimension vector and $R$ is a commutative $k$-algebra, then we view $M \in \bbA^{\bd, \bd} (R)$ as an object of $\mod_{k \Delta} (R)$, where $M_x := R^{d (x)}$, for each $x \in \Delta_0$. Consequently, we define $\mod_\Lambda^\bd$ as the subscheme of $\bbA^{\bd, \bd}$ such that, for a commutative $k$-algebra $R$, $\mod_\Lambda^\bd (R)$ consists of $M \in \bbA^{\bd, \bd} (R)$ which are $R$-representations of $(\Delta, \frakR)$. Then $\mod_\Lambda^\bd$ is an affine scheme, which is called the scheme of $\Lambda$-modules of dimension vector $\bd$, as for each $N \in \mod \Lambda$ with dimension vector $\bd$, there exists $M \in \mod_\Lambda^\bd (k)$ such that $M \simeq N$. The scheme $\mod_\Lambda^\bd$ is described within $\bbA^{\bd, \bd}$ by $\sum_{\rho \in \frakR} d (s \rho) d (t \rho)$ equations. Consequently, if
\[
a_\Lambda (\bd) := \sum_{\alpha \in \Delta_1} d (s \alpha) d (t \alpha) - \sum_{\rho \in \frakR} d (s \rho) d (t \rho)
\]
and $\calZ$ is an irreducible component of $\mod_\Lambda^\bd (k)$, then $\dim \calZ \geq a_\Lambda (\bd)$. Note that
\begin{equation} \label{eq ad}
a_\Lambda (\bd) = \dim_k \bbV^{\bd, \bd} (k) - \chi_\Lambda (\bd),
\end{equation}
provided $\Lambda$ is triangular and $\gldim \Lambda \leq 2$.

For a dimension vector $\bd$, let $\GL_\bd$ be the affine group scheme defined by
\[
\GL_\bd := \prod_{i \in \Delta_0} \GL_{d_i}.
\]
Then $\GL_\bd$ acts on $\bbA^{\bd, \bd}$ by
\[
(g \ast M)_\alpha := g_{t \alpha} M_\alpha g_{s \alpha}^{-1},
\]
for $g \in \GL_\bd (R)$, $M \in \bbA^{\bd, \bd} (R)$, and $\alpha \in \Delta_1$, where $R$ is a commutative $k$-algebra. Note that $\mod_\Lambda^\bd$ is a $\GL_\bd$-invariant subscheme of $\bbA^{\bd, \bd}$. If $M \in \mod_\Lambda^\bd (k)$, then we denote by $\calO_M$ the $\GL_\bd (k)$-orbit of $M$ in $\mod_\Lambda^\bd (k)$, and by $\ol{\calO}_M$ the closure of $\calO_M$. If $\calZ$ is a $\GL_\bd (k)$-invariant subset of $\mod_\Lambda^\bd (k)$ and $M \in \calZ$, then we say that $\calO_M$ is maximal in $\calZ$ if there is no $N \in \calZ$ such that $\calO_M \subseteq \ol{\calO}_N$ and $\calO_M \neq \calO_N$. Note that $\calO_M = \calO_N$ if and only if $M \simeq N$. 

We present now an interpretation of the extension spaces, which will play an important role later. Let $R$ be a commutative $k$-algebra, $\bd_1$ and $\bd_2$ dimension vectors, $M \in \mod_\Lambda^{\bd_1} (R)$ and $N \in \mod_\Lambda^{\bd_2} (R)$, and fix $Z \in \bbA^{\bd_2, \bd_1} (R)$. If $\sigma = \alpha_1 \cdots \alpha_n$ is a path in $\Delta$ with $\alpha_1, \ldots, \alpha_n \in \Delta_1$, then we put
\[
Z_\sigma^{N, M} := \sum_{i = 1}^n M_{\alpha_1} \cdots M_{\alpha_{i - 1}} Z_{\alpha_i} N_{\alpha_{i + 1}} \cdots N_{\alpha_n}.
\]
Moreover, if $\rho = \lambda_1 \sigma_1 + \cdots + \lambda_n \sigma_n$ is a relation, for $\lambda_1, \ldots, \lambda_n \in k$ and paths $\sigma_1$, \ldots, $\sigma_n$, then
\[
Z_\rho^{N, M} := \lambda_1 Z_{\sigma_1}^{N, M} + \cdots + \lambda_n Z_{\sigma_n}^{N, M}.
\]
We denote by $\bbZ^{N, M}$ the set of $Z \in \bbA^{\bd_2, \bd_1} (R)$ such that $Z_\rho^{N, M}  = 0$, for all $\rho \in \frakR$. If $Z \in \bbZ^{N, M}$, then we define $W^Z \in \mod_\Lambda^{\bd_1 + \bd_2} (R)$ by
\[
W^Z :=
\begin{bmatrix}
M & Z
\\
0 & N
\end{bmatrix}
\qquad \left( \text{i.e. }
W^Z_\alpha :=
\begin{bmatrix}
M_\alpha & Z_\alpha
\\
0 & N_\alpha
\end{bmatrix}
\qquad (\alpha \in \Delta_1) \right).
\]
Then we have the canonical exact sequence
\[
\xi^Z \colon 0 \to M \to W^Z \to N \to 0.
\]
The assignment
\[
\bbZ^{N, M} \ni Z \mapsto [\xi^Z] \in \Ext_\Lambda^1 (N, M)
\]
is a linear epimorphism. We denote its kernel by $\bbB^{N, M}$. Then $Z \in \bbB^{N, M}$ if and only if there exists $h \in \bbV^{\bd_2, \bd_1} (R)$ such that
\[
Z = h \circ N - M \circ h.
\]
Consequently, we have exact sequences
\begin{gather}
\nonumber 0 \to \Hom_\Lambda (N, M) \to \bbV^{\bd_2, \bd_1} (R) \to \bbB^{N, M} \to 0
\\
\intertext{and}
0 \to \bbB^{N, M} \to \bbZ^{N, M} \to \Ext_\Lambda^1 (N, M) \to 0. \label{seqBZExt}
\end{gather}
In particular, if $M \in \mod_\Lambda^{\bd_1} (k)$ and $N \in \mod_\Lambda^{\bd_2} (k)$, then
\begin{gather}
\dim_k \bbB^{N, M} = \dim_k \bbV^{\bd_2, \bd_1} (k) - [N, M]_\Lambda \nonumber
\\
\intertext{and}
\dim_k \bbZ^{N, M} = \dim_k \bbV^{\bd_2, \bd_1} (k) - [N, M]_\Lambda + [N, M]_\Lambda^1. \label{eqdimZ}
\end{gather}

The next thing we need is a relationship of elements of $\mod_\Lambda^{\bd} (k [t] / t^n)$ with exact sequences. Fix $M \in \mod_\Lambda^\bd (k [t] / t^n)$. Put $U := M / t M \in \mod_\Lambda^{\bd} (k)$ and $V := M / t^{n - 1} M \in \mod_\Lambda^{\bd} (k [t] / t^{n - 1})$. If we write
\[
M = M_0 + t M_1 + \cdots + t^{n - 1} M_{n - 1},
\]
for $M_0, \ldots, M_{n - 1} \in \bbA^{\bd, \bd} (k)$, then $M_0 = U$ and $V = M_0 + t M_1 + \cdots + t^{n - 2} M_{n - 2}$. We may view $M$ and $V$ as $\Lambda$-modules of dimension vectors $n \bd$ and $(n - 1) \bd$, respectively. Then we have isomorphisms
\[
M \simeq
\begin{bmatrix}
U & M_1 & \cdots & M_{n - 1}
\\
0 & U & \ddots & \vdots
\\
\vdots & & \ddots & M_1
\\
0 & \cdots & 0 & U
\end{bmatrix}
\qquad \text{and} \qquad
V \simeq
\begin{bmatrix}
U & M_1 & \cdots & M_{n - 2}
\\
0 & U & \ddots & \vdots
\\
\vdots & & \ddots & M_1
\\
0 & \cdots & 0 & U
\end{bmatrix}
\]
of $\Lambda$-modules, and consequently exact sequences
\[
0 \to U \to M \to V \to 0 \qquad \text{and} \qquad 0 \to V \to M \to U \to 0.
\]

The above gives us an interpretation of the tangent spaces. Let $k [\varepsilon]$ be the algebra of dual numbers and $\pi \colon k [\varepsilon] \to k$ the canonical projection. Recall that if $M \in \mod_\Lambda^\bd (k)$, then the tangent space $T_M \mod_\Lambda^\bd$ is $(\mod_\Lambda^\bd (\pi))^{-1} (M)$. If $M + \varepsilon Z \in T_M \mod_\Lambda^\bd$, then the above shows that $Z \in \bbZ^{M, M}$, and we identify $T_M \mod_\Lambda^\bd$ with $\bbZ^{M, M}$. Under this identification, $\bbB^{M, M} = T_M \calO_M$ and~\eqref{seqBZExt} takes the form
\[
0 \to T_M \calO_M \to T_M \mod_\Lambda^\bd \to \Ext_\Lambda^1 (M, M) \to 0,
\]
which is the famous Voigt's result~\cite{Voigt}.

We finish this section with the following variant of the above constructions. Let $\bd_1$ and $\bd_2$ be dimension vectors and $\bd := \bd_1 + \bd_2$. Then we have a natural decomposition
\[
\bbA^{\bd, \bd} =
\begin{bmatrix}
\bbA^{\bd_1, \bd_1} & \bbA^{\bd_2, \bd_1}
\\
\bbA^{\bd_1, \bd_2} & \bbA^{\bd_2, \bd_2}
\end{bmatrix}.
\]
We define $\calE_\Lambda^{\bd_2, \bd_1}$ to be the intersection of $\mod_\Lambda^\bd$ and $\left[
\begin{smallmatrix}
\bbA^{\bd_1, \bd_1} & \bbA^{\bd_2, \bd_1}
\\
0 & \bbA^{\bd_2, \bd_2}
\end{smallmatrix}
\right]$. Consequently, if $R$ is a commutative $k$-algebra, then
\begin{multline*}
\calE_\Lambda^{\bd_2, \bd_1} (R) =
\biggl\{
\begin{bmatrix}
U & Z
\\
0 & V
\end{bmatrix} :
\\
\text{$U \in \mod_\Lambda^{\bd_1} (R)$, $V \in \mod_\Lambda^{\bd_2} (R)$, and $Z \in \bbZ^{V, U}$}  \biggr\}.
\end{multline*}
In particular, if $U \in \mod_\Lambda^{\bd_1} (k)$, $V \in \mod_\Lambda^{\bd_2} (k)$, and $N := U \oplus V$, where we identify $U \oplus V$ with the element $\left[
\begin{smallmatrix}
U & 0
\\
0 & V
\end{smallmatrix}
\right]$ of $\mod_\Lambda^\bd (k)$, then
\[
T_N \calE_\Lambda^{\bd_2, \bd_1} =
\begin{bmatrix}
\bbZ^{U, U} & \bbZ^{V, U}
\\
0 & \bbZ^{V, V}
\end{bmatrix}.
\]
If additionally $h, e \in \bbN$, then one defines (see~\cite{Bobinski2009} for details) the scheme $\calE_{\Lambda, h, e}^{\bd_2, \bd_1}$ such that
\[
\calE_{\Lambda, h, e}^{\bd_2, \bd_1} (k) =
\biggl\{
\begin{bmatrix}
U & Z
\\
0 & V
\end{bmatrix} \in \calE_\Lambda^{\bd_2, \bd_1} (k) : \text{$[V, U]_\Lambda = h$ and $[V, U]_\Lambda^1 = e$}  \biggr\}.
\]
Recall that if $Z \in \bbZ^{M, N}$, then $\xi^Z$ denotes the corresponding exact sequence. Moreover, if $\xi_1 \in \Ext_\Lambda^n (M, N)$ and $\xi_2 \in \Ext_\Lambda^m (N, L)$, then $\xi_2 \circ \xi_1 \in \Ext_\Lambda^{n + m} (M, L)$ is the Yoneda product. The following is a reformulation of~\cite{Bobinski2009}*{Propositions~3.2 and~3.3}.

\begin{proposition} \label{prop tangent}
Let $\bd_1$ and $\bd_2$ be dimension vectors, $U \in \mod_\Lambda^{\bd_1} (k)$, $V \in \mod_\Lambda^{\bd_2} (k)$, and $N := U \oplus V$. If
\[
h := \dim_k \Hom_\Lambda (V, U) \qquad \text{and} \qquad e := \dim_k \Ext_\Lambda^1 (V, U),
\]
and
\[
Z =
\begin{bmatrix}
Z_{1, 1} & Z_{2, 1}
\\
0 & Z_{2, 2}
\end{bmatrix} \in T_N \calE_\Lambda^{\bd_2, \bd_1},
\]
then $Z \in T_N \calE_{\Lambda, h, e}^{\bd_2, \bd_1}$ if and only if
\begin{enumerate}

\item
$[\xi^{Z_{1, 1}}] \circ f = f \circ [\xi^{Z_{2, 2}}]$, for each $f \in \Hom_\Lambda (V, U)$, and

\item
$[\xi^{Z_{1, 1}}] \circ \xi + \xi \circ [\xi^{Z_{2, 2}}] = 0$, for each $\xi \in \Ext_\Lambda^1 (V, U)$.

\end{enumerate}
\end{proposition}

\section{Criterion for nonsingularity} \label{sect nonsingularity}

We first prove the following result, which is crucial for the proof of the main result.

\begin{proposition} \label{prop reduction}
Let $\bd_1$ and $\bd_2$ be dimension vectors, $\bd := \bd_1 + \bd_2$, $U \in \mod_\Lambda^{\bd_1} (k)$ and $V \in \mod_\Lambda^{\bd_2} (k)$. Let $N := U \oplus V$ and $\calZ$ be a $\GL_\bd (k)$-invariant irreducible closed subset of $\mod_\Lambda^\bd (k)$, considered as an affine scheme with its reduced structure, such that $N \in \calZ$. If $[U, V]_\Lambda^1 = 0$, then the canonical map
\[
\GL_\bd \times (\calZ \cap \calE_\Lambda^{\bd_2, \bd_1}) \to \calZ
\]
is smooth at $(\Id, N)$. In particular, the scheme $\calZ \cap \calE_\Lambda^{\bd_2, \bd_1}$ is reduced at $N$ and
\[
\dim_k T_N \calZ = \dim_k T_N (\calZ \cap \calE_\Lambda^{\bd_2, \bd_1}) + \dim_k \bbB^{U, V}.
\]
\end{proposition}

\begin{proof}
Fix a (linear) complement $\calC_0$ of $\bbB^{U, V}$ in $\bbA^{\bd_1, \bd_2} (k)$. Put
\[
\calC :=
\begin{bmatrix}
\bbA^{\bd_1, \bd_1} (k) & \bbA^{\bd_2, \bd_1} (k)
\\
\calC_0 & \bbA^{\bd_2, \bd_2} (k)
\end{bmatrix},
\]
which we consider as an affine scheme with the reduced structure (given by the equations describing $\calC_0$). Using general properties of smooth morphisms (compare~\cite{Slodowy}*{Section~5.1}), it follows that the canonical map
\[
\GL_\bd \times (\calZ \cap \calC) \to \calZ
\]
is smooth at $(\Id, N)$. Obviously, $\calE_\Lambda^{\bd_2, \bd_1} \subseteq \calC$, hence $\calZ \cap \calE_\Lambda^{\bd_2, \bd_1} \subseteq \calZ \cap \calC$ (both inclusions are inclusions of schemes). We prove that there exists an open neighbourhood $\calU$ of $N$ in $\calZ \cap \calC$ such that $\calU \subseteq \calE_\Lambda^{\bd_2, \bd_1} (k)$. This will imply that the schemes $\calZ \cap \calC$ and $\calZ \cap \calE_\Lambda^{\bd_2, \bd_1}$ coincide in a neighborhood of $N$, hence finish the proof.

We show that every irreducible component $\calX$ of $\calZ \cap \calC$ containing $N$ is contained in $\calE_\Lambda^{\bd_2, \bd_1} (k)$. Assume this is not the case and let $\calX$ be an irreducible component of $\calZ \cap \calC$ containing $N$ such that $\calX \not \subseteq \calE_\Lambda^{\bd_2, \bd_1} (k)$. Note that $\calX$ is $\GL_{\bd_1} (k) \times \GL_{\bd_2} (k)$-invariant, since both $\calZ$ and $\calC$ are $\GL_{\bd_1} (k) \times \GL_{\bd_2} (k)$-invariant, where we embed $\GL_{\bd_1} \times \GL_{\bd_2}$ in $\GL_\bd$ diagonally. If $\calU := \calX \setminus \calE_\Lambda^{\bd_2, \bd_1} (k)$, then $\calU$ is a nonempty open subset of $\calX$, hence $\calX = \ol{\calU}$. Using basic facts from algebraic geometry one shows there exists a nonsingular curve $\calC$ and a regular map $\varphi \colon \calC \to \calX$ such that $\varphi (c_0) = N$, for some $c_0 \in \calC$, and $\varphi^{-1} (\calU)$ is a cofinite subset of $\calC$ (see for example~\cite{Zwara2011}*{Lemma~5.3}). Let $R$ be the local ring of $\calC$ at $c_0$ and $t$ its uniformizing element. A map $\varphi$ is represented by $A \in \mod_\Lambda^\bd (R)$ such that $A / t A = N$.

Write
\[
A =
\begin{bmatrix}
A_{1, 1} & A_{2, 1} \\ A_{1, 2} & A_{2, 2}
\end{bmatrix},
\]
for $A_{i, j} \in \bbA^{\bd_i, \bd_j} (R)$, $i, j \in \{ 1, 2 \}$. In particular,
\begin{gather*}
A_{1, 1} / t A_{1, 1} = U, \qquad A_{1, 2} / t A_{1, 2} = 0,
\\
A_{2, 1} / t A_{2, 1} = 0 \qquad \text{and} \qquad A_{2, 2} / t A_{2, 2} = V.
\end{gather*}
Since $\varphi^{-1} (\calU)$ is a cofinite subset of $\calC$, $A_{1, 2} \neq 0$. Let $n \in \bbN$ be maximal such that $t^n \mid A_{1, 2}$. Then $n > 0$. Let $g \in \GL_\bd (K)$, where $K$ is the field of fractions of $R$, be given by
\[
g :=
\begin{bmatrix}
t^{n - 1} \Id & 0 \\ 0 & \Id
\end{bmatrix}.
\]
Then
\[
B := g \ast A =
\begin{bmatrix}
A_{1, 1} & t^{n - 1} A_{2, 1} \\ t \frac{A_{1, 2}}{t^n} & A_{2, 2}
\end{bmatrix}.
\]
By shrinking $\calC$ if necessary, we get from $B$ a regular map $\psi \colon \calC \to \calX$ (we use here that $\calX$ is $\GL_{\bd_1} (k) \times \GL_{\bd_2} (k)$-invariant and closed) such that $\psi (c_0) = N$. The maximality of $n$ implies that
\[
\Im (T_{c_0} \psi) \not \subseteq
\begin{bmatrix}
\bbA^{\bd_1, \bd_1} & \bbA^{\bd_2, \bd_1} \\ 0 & \bbA^{\bd_2, \bd_2}
\end{bmatrix}.
\]
On the other hand,
\[
T_N (\calZ \cap \calC) \subseteq (T_N \mod_\Lambda^\bd) \cap \calC =
\begin{bmatrix}
\bbZ^{U, U} & \bbZ^{V, U} \\ \bbZ^{U, V} & \bbZ^{V, V}
\end{bmatrix}
\cap \calC \subseteq
\begin{bmatrix}
\bbA^{\bd_1, \bd_1} & \bbA^{\bd_2, \bd_1} \\ 0 & \bbA^{\bd_2, \bd_2}
\end{bmatrix},
\]
since $\bbZ^{U, V} = \bbB^{U, V}$ (as $[U, V]_\Lambda^1 = 0$). This gives a contradiction, which finishes the proof.
\end{proof}

If $\calV$ is a subset of $\calE_\Lambda^{\bd_2, \bd_1} (k)$, for dimension vectors $\bd_1$ and $\bd_2$, then
\begin{gather*}
\hom (\calV) := \min \{ [V, U]_\Lambda : \text{$(U, V) \in \pi (\calV)$} \}
\\
\intertext{and}
\ext (\calV) := \min \{ [V, U]_\Lambda^1 : \text{$(U, V) \in \pi (\calV)$} \},
\end{gather*}
where $\pi \colon \calE_\Lambda^{\bd_2, \bd_1} \to \mod_\Lambda^{\bd_1} \times \mod_\Lambda^{\bd_2}$ is the canonical projection. We will use the following consequence of Proposition~\ref{prop reduction}.

\begin{corollary} \label{coro nonsingular}
Assume $\Lambda$ is triangular and $\gldim \Lambda \leq 2$. Let $\bd_1$ and $\bd_2$ be dimension vectors, $U \in \mod_\Lambda^{\bd_1} (k)$, $V \in \mod_\Lambda^{\bd_2} (k)$, $\bd := \bd_1 + \bd_2$, and $N := U \oplus V$. Let $\calZ$ be an irreducible component of $\mod_\Lambda^\bd (k)$, considered as an affine scheme with its reduced structure, such that $N \in \calZ$.
Put
\[
h := [V, U]_\Lambda \qquad \text{and} \qquad e := [V, U]_\Lambda^1.
\]
Assume $\Ext_\Lambda^1 (U, V) = 0$ and there exists an open subset $\calV$ of $\calZ \cap \calE_\Lambda^{\bd_2, \bd_1} (k)$ such that $N \in \calV$ and
\[
\hom (\calV) = h \qquad and \qquad \ext (\calV) = e.
\]
If $\idim_\Lambda V \leq 1$ and there exists an exact sequence
\[
\xi \colon 0 \to U \to M \to V \to 0
\]
with $\pdim_\Lambda M \leq 1$, then $N$ is a nonsingular point of $\calZ$.
\end{corollary}

\begin{proof}
We want to show that $\dim_k T_N \calZ \leq \dim \calZ$. Since $\dim \calZ \geq a_\Lambda (\bd)$, it is enough to show that $\dim_k T_N \calZ \leq a_\Lambda (\bd)$. Using Proposition~\ref{prop reduction} it is sufficient to find an appropriate upper bound for the dimension of $T_N (\calZ \cap \calE_\Lambda^{\bd_2, \bd_1})$. Since the scheme $\calZ \cap \calE_\Lambda^{\bd_2, \bd_1}$ is reduced at $N$ by Proposition~\ref{prop reduction}, our assumptions imply that $T_N (\calZ \cap \calE_\Lambda^{\bd_2, \bd_1}) = T_N \calV \subseteq T_N \calE_{\Lambda, h, e}^{\bd_2, \bd_1}$.

We first observe that $\pdim_\Lambda U \leq 1$. Indeed, if $X$ is a $\Lambda$-module and we apply $\Hom_\Lambda (-, X)$ to $\xi$, then we get an exact sequence
\[
0 = \Ext_\Lambda^2 (M, X) \to \Ext_\Lambda^2 (U, X) \to \Ext_\Lambda^3 (V, X) = 0,
\]
as $\pdim_\Lambda M \leq 1$ and $\gldim \Lambda \leq 2$. Since $\pdim_\Lambda U \leq 1$ and $\idim_\Lambda V \leq 1$, $\Ext_\Lambda^2 (N, N) = \Ext_\Lambda^2 (V, U)$. Let
\[
Z =
\begin{bmatrix}
Z_{1, 1} & Z_{2, 1} \\ 0 & Z_{2, 2}
\end{bmatrix}
\in T_N \calE_{\Lambda, h, e}^{\bd_1, \bd_2}.
\]
Then $Z_{1, 1} \in \bbZ^{U, U}$, $Z_{2, 1} \in \bbZ^{V, U}$, $Z_{2, 2} \in \bbZ^{V, V}$, and $[\xi^{Z_{1, 1}}] \circ [\xi] + [\xi] \circ [\xi^{Z_{2, 2}}] = 0$, by Proposition~\ref{prop tangent}. If we apply $\Hom_\Lambda (-, U)$ to $\xi$, then we get an exact sequence
\[
\Ext_\Lambda^1 (U, U) \to \Ext_\Lambda^2 (V, U) \to \Ext_\Lambda^2 (M, U) = 0.
\]
Consequently, the map
\[
\bbZ^{U, U} \times \bbZ^{V, V} \to \Ext_\Lambda^2 (V, U), \; (Z_{1, 1}, Z_{2, 2}) \mapsto [\xi^{Z_{1, 1}}] \circ [\xi] + [\xi] \circ [\xi^{Z_{2, 2}}],
\]
is an epimorphism, hence
\[
\dim_k T_N \calE_{\Lambda, h, e}^{\bd_1, \bd_2} \leq \dim_k \bbZ^{U, U} + \dim_k \bbZ^{V, V} + \dim_k \bbZ^{V, U} - \dim_k \Ext_\Lambda^2 (V, U).
\]
Since $\bbZ^{U, V} = \bbB^{U, V}$ and $\Ext_\Lambda^2 (V, U) = \Ext_\Lambda^2 (N, N)$, Proposition~\ref{prop reduction} implies
\[
\dim_k T_N \calZ \leq \dim_k \bbZ^{N, N} - \dim_k \Ext_\Lambda^2 (N, N).
\]
Using~\eqref{eqdimZ}, we get
\[
\dim_k T_N \calZ \leq \dim_k \bbV^{\bd, \bd} (k) - \chi_\Lambda (\bd) = a_\Lambda (\bd),
\]
where the latter equality follows from~\eqref{eq ad}, and the claim follows.
\end{proof}

\section{Existence of an exact sequence} \label{sect sequence}

In view of Corollary~\ref{coro nonsingular}, the following result is important.

\begin{proposition} \label{prop main}
Let $\bd_1$ and $\bd_2$ be dimension vectors, $U \in \mod_\Lambda^{\bd_1} (k)$, $V \in \mod_\Lambda^{\bd_2} (k)$, $\bd := \bd_1 + \bd_2$ and $N := U \oplus V$. Let $\calZ$ be a closed irreducible subset of $\mod_\Lambda^\bd (k)$ and let $\calU$ be a nonempty open subset of $\calZ$ such that $N \in \calZ \setminus \calU$ and $\calO_N$ is maximal in $\calZ \setminus \calU$. Assume $\Ext_\Lambda^1 (U, V) = 0$ and there exists a $\Lambda$-module $H$ such that
\begin{equation} \label{eq assumption}
[H, V]_\Lambda > \min \{ [H, M]_\Lambda : \text{$M \in \calZ$} \}
\end{equation}
and $\Ext_\Lambda^1 (H, V) = 0$. Then there exist an exact sequence
\[
0 \to U \to M \to V \to 0
\]
with $M \in \calU$.
\end{proposition}

The rest of the section is devoted to the proof of Proposition~\ref{prop main}. We divide the proof into steps.

Let
\begin{equation} \label{eq U0}
\calU_0 := \{ M \in \calZ : \text{$[H, M]_\Lambda = d$} \},
\end{equation}
where $d := \min \{ [H, M]_\Lambda : \text{$M \in \calZ$} \}$. Then $\calU_0$ is an open subset of $\calZ$ and $\calZ = \ol{\calU_0}$. Since $N \in \calZ = \ol{\calU_0}$, there exists a nonsingular curve $\calC$ and a regular map $\varphi \colon \calC \to \calZ$ such that $\varphi (c_0) = N$, for some $c_0 \in \calC$, and $\varphi^{-1} (\calU_0)$ is a cofinite subset of $\calC$. Let $R$ be the local ring of $\calC$ at $c_0$ and $t$ its uniformizing element. A map $\varphi$ is represented by $A \in \mod_\Lambda^\bd (R)$ such that $A / t A = N$. The next observation is inspired by~\cite{Zwara2000}*{Lemma~3.3}.

\begin{step} \label{step modify}
There exist $n \in \{ 2, 3, \ldots \} \cup \{ \infty \}$ and $g_i \in \GL_\bd (R)$, $0 \leq i < n$, such that the following conditions are satisfied:
\begin{enumerate}

\item
If $1 \leq i < n$ and $A_i := g_{i - 1} \ast A$, then
\[
A_i / t^i A_i  = U_i \oplus V_i,
\]
for some $U_i \in \mod_\Lambda^{\bd_1} (R / t^i)$ and $V_i \in \mod_\Lambda^{\bd_2} (R / t^i)$. Moreover, $U_1 = U$, $V_1 = V$, and if $1 < i < n$, then
\[
U_i / t^{i - 1} U_i = U_{i - 1} \qquad \text{and} \qquad V_i / t^{i - 1} V_i = V_{i - 1}.
\]

\item
If $n < \infty$ and $A_n := g_{n - 1} \ast A$, then
\[
A_n / t^n A_n =
\begin{bmatrix}
U_n & t^{n - 1} Z
\\
0 & V_n
\end{bmatrix},
\]
for some $U_n \in \mod_\Lambda^{\bd_1} (R / t^n)$, $V_n \in \mod_\Lambda^{\bd_2} (R / t^n)$, and $Z \in \bbZ^{V, U} \setminus \bbB^{V, U}$, such that
\[
U_n / t^{n - 1} U_n = U_{n - 1} \qquad \text{and} \qquad V_n / t^{n - 1} V_n = V_{n - 1}.
\]

\end{enumerate}
\end{step}

\begin{proof}
We prove the claim by induction on $i$. We take $g_0 := \Id$, hence the claim follows for $i = 1$. Assume $n$ is not yet defined, $i > 1$, $g_0$, \ldots, $g_{i - 2}$ are already constructed, and $A_1$, \ldots, $A_{i - 1}$, defined as above, satisfy the above conditions. Then
\[
A_{i - 1} / t^i A_{i - 1} =
\begin{bmatrix}
U_i' & t^{i - 1} Z
\\
t^{i - 1} Z' & V_i'
\end{bmatrix},
\]
for $U_i' \in \mod_\Lambda^{\bd_1} (R / t^i)$, $V_i' \in \mod_\Lambda^{\bd_2} (R / t^i)$, $Z \in \bbA^{\bd_2, \bd_1} (k)$ and $Z' \in \bbA^{\bd_1, \bd_2} (k)$, such that
\[
U_i' / t^{i - 1} U_i' = U_{i - 1} \qquad \text{and} \qquad V_i' / t^{i - 1} V_i' = V_{i - 1}.
\]
If $\rho$ is a relation in $\Delta$, then one calculates that the lower-left coefficient of $(A_{i - 1} / t^i A_{i - 1})_\rho$ is $t^{i - 1} Z'^{U, V}_\rho$. Consequently, $Z' \in \bbZ^{U, V}$. Since $\Ext_\Lambda^1 (U, V) = 0$, $\bbZ^{U, V} = \bbB^{U, V}$, hence there exists $h' \in \bbV^{\bd_1, \bd_2} (k)$ such that
\[
Z' = h' \circ U - V \circ h'.
\]
We put
\[
g' :=
\begin{bmatrix}
\Id & 0
\\
- t^{i - 1} h' & \Id
\end{bmatrix}
\in \GL_\bd (R) .
\]
If $A' := g' \ast A_{i - 1}$, then
\[
A' / t^i A' =
\begin{bmatrix}
U_i' & t^{i - 1} Z
\\
0 & V_i'
\end{bmatrix}.
\]
Again $Z \in \bbZ^{V, U}$. If $Z \in \bbB^{V, U}$, then similarly as above we find $g'' \in \GL_\bd (R)$ such that the claim follows for $i$, provided $g_{i - 1} := g'' g' g_{i - 2}$, hence we may proceed by induction. Otherwise, we put $n := i$ and $g_{i - 1} := g' g_{i - 2}$, and finish the proof.
\end{proof}

If we are in the latter case, we obtain our claim.

\begin{step}
Assume there exists $g \in \GL_\bd (R)$ and $n \in \bbN_+$ such that, if $B := g \ast A$, then
\[
B / t B = U \oplus V \qquad \text{and} \qquad B / t^{n + 1} B =
\begin{bmatrix}
U' & t^n Z
\\
0 & V'
\end{bmatrix},
\]
for some $U' \in \mod_\Lambda^{\bd_1} (R / t^{n + 1})$, $V' \in \mod_\Lambda^{\bd_2} (R / t^{n + 1})$, and $Z \in \bbZ^{V, U} \setminus \bbB^{V, U}$. There there exists an exact sequence
\[
0 \to U \to M \to V \to 0
\]
with $M \in \calU$.
\end{step}

\begin{proof}
Put
\[
M :=
\begin{bmatrix}
U & Z
\\
0 & V
\end{bmatrix}.
\]
Then there exists an exact sequence
\[
0 \to U \to M \to V \to 0.
\]
In particular, $\calO_N \subseteq \ol{\calO}_M$ (see for example~\cite{Bongartz1996}*{Section~1.1}). Moreover, $M \not \simeq N$ (since $Z \not \in \bbB^{V, U}$), hence $\calO_N \neq \calO_M$.

Put $g' := \left[
\begin{smallmatrix}
\Id & 0
\\
0 & t^n \Id
\end{smallmatrix}
\right] \in \GL_\bd (K)$, where $K$ is the field of fractions of $R$, and $B' := g' \ast B$. Then $B' \in \mod_\Lambda^\bd (R)$ (although, $g' \not \in \GL_\bd (R)$) and $B' / t B' = M$. Hence, there exists a cofinite subset $\calC'$ of $\calC$ such that $c_0 \in \calC$ and $B'$ defines a regular map $\psi \colon \calC' \to \calZ$ (we use that $\calZ$ is $\GL_\bd (k)$-invariant and closed) with $\psi (c_0) = M$. In particular, $M \in \calZ$. Since $\calO_N$ is maximal in $\calZ \setminus \calU$, $M \in \calU$, and the claim follows.
\end{proof}

The following observation finishes the proof.

\begin{step}
There do not exist $g_i \in \GL_\bd (R)$, $i \in \bbN_+$, such that the following conditions are satisfies, where, for $i \in \bbN_+$, $A_i := g_i \ast A$:
\begin{enumerate}

\item
if $i \in \bbN_+$, then $A_i / t^i A_i = U_i \oplus V_i$, for some $U_i \in \mod_\Lambda^{\bd_1} (R / t^i)$ and $V_i \in \mod_\Lambda^{\bd_2} (R / t^i)$, and

\item
$U_1 = U$, $V_1 = V$, and, if $i \in \bbN_+$, then
\[
U_{i + 1} / t^{i + 1} U_{i + 1} = U_i \qquad \text{and} \qquad V_{i + 1} / t^{i + 1} V_{i + 1} = V_i.
\]

\end{enumerate}
\end{step}

\begin{proof}
Assume that such $g_i$, $i \in \bbN_+$, exist. We show this leads to a contradiction. Recall that $U_i$ and $V_i$ can be viewed as $\Lambda$-modules of dimension vectors $i \bd_1$ and $i \bd_2$, respectively. Moreover, for each $i \in \bbN_+$, we have an exact sequence
\begin{equation} \label{eq sequence}
0 \to V \to V_{i + 1} \to V_i \to 0.
\end{equation}

For each $i \in \bbN_+$, let $\varphi_i \colon \calC_i \to \mod_\Lambda^\bd (k)$ be the regular map defined by $A_i$, where $\calC_i$ is a cofinite subset of $\calC$ containing $c_0$. Then, for each $i \in \bbN_+$, $\varphi_i (c_0) = N$ and $\varphi_i^{-1} (\calU_0)$ is a cofinite subset of $\calC_i$.

If $W \in \mod_\Lambda^\bd (S)$, for a commutative ring $S$, then $\Hom_\Lambda (H, W)$ is the solution of the set of $p$ linear homogeneous equations
\[
W_\alpha f_{s \alpha} - f_{t \alpha} H_\alpha = 0, \qquad \alpha \in \Delta_1,
\]
with coefficients in $S$ and in $q$ indeterminates, which form an element $f$ of $\bbV^{\bh, \bd} (S)$, where
\[
p := \sum_{\alpha \in \Delta_1} h (s \alpha) d (t \alpha), \qquad q := \sum_{y \in \Delta_0} h (y) d (y),
\]
and $\bh := \bdim H$. If we associate with $W$ the matrix $\Phi (W)$ of this system, we obtain a morphism $\Phi \colon \mod_\Lambda^\bd \to \bbM_{p, q}$ of schemes. We treat $\bbM_{p, q}$ as the scheme of representations of the one arrow quiver (with no relations) of dimension vector $(p, q)$. The map $\Phi$ is equivariant in the following sense: there exists a morphism $\Psi \colon \GL_\bd \to \GL_{(p, q)}$ of algebraic groups such that
\[
\Phi (g \ast W) = \Psi (g) \ast \Phi (W),
\]
for all $g \in \GL_\bd$ and $W \in \mod_\Lambda^\bd$.

For $r \in [0, \min (p, q)]$, let $\calO_r$ be the $\GL_{(p, q)} (k)$-orbit in $\bbM_{p, q} (k)$ consisting of the matrices of rank $r$. The equation~\eqref{eq U0} implies that $\Phi (M) \in \calO_{q - d}$, if $M \in \calU_0$.

For $i \in \bbN_+$, let $\psi_i := \Phi \circ \varphi_i$ and $B_i := \Phi (A_i)$ be the corresponding element of $\bbM_{p, q} (R)$. Note that $B_i \simeq B := B_1$, for all $i \in \bbN_+$, since the map $\Phi$ is  equivariant. Moreover $\psi_i^{-1} (\calO_{q - d})$ is a cofinite subset of $\calC_i$, hence $B_i \otimes_R K \simeq L \otimes_k K$, where $L$ is a chosen element of $\calO_{q - d}$. Consequently, by results of~\cite{Zwara2000}*{(3.3), (3.5)}, there exists $l$ such that
\begin{equation} \label{eq isomorphism}
B / t^{j + 1} B \simeq (B / t^j B) \oplus L,
\end{equation}
for all $j \geq l$.

Fix $j \geq l$. We calculate $[H, N_j]_\Lambda$, where $N_j := A / t^j A$. Our assumptions imply $N_j \simeq A_j / t^j A_j = U_j \oplus V_j$. In particular, $[H, N_j]_\Lambda \geq [H, V_j]_\Lambda$. By applying $\Hom_\Lambda (N, -)$ to the sequence~\eqref{eq sequence}, we get the exact sequence
\[
0 \to \Hom_\Lambda (H, V) \to \Hom_\Lambda (H, V_{j + 1}) \to \Hom_\Lambda (H, V_j) \to 0,
\]
since $\Ext_\Lambda^1 (H, V) = 0$. Consequently, by easy induction
\begin{equation} \label{eq dimension}
[H, N_j]_\Lambda \geq [H, V_j]_\Lambda = j \cdot [H, V]_\Lambda \geq j d + j,
\end{equation}
since $[H, V]_\Lambda > d$ by assumption~\eqref{eq assumption}. On the other hand,
\[
[H, N_j] = q j - \rank_k (B / t^j B),
\]
where $\rank_k (B / t^j B)$ denotes the $k$-dimension of the image of the map $(R / t^j)^q \to (R / t^j)^p$ determined by $B / t^j B$. The equation~\eqref{eq isomorphism} implies,
\[
B / t^j B \simeq (B / t^l B) \oplus L^{j - l}.
\]
Since $L \in \calO_{q - d}$,
\begin{multline*}
[H, N_j]_\Lambda = q j - \rank_k (B / t^l B) - (q - d) (j - l)
\\
= q l - \rank_k (B / t^l B) + j d - l d = j d + ([H, N_l]_\Lambda - l d).
\end{multline*}
If $j > [H, N_l]_\Lambda - l d$, we get a contradiction with~\eqref{eq dimension}.
\end{proof}

\section{Main result} \label{section main result}

Now we describe the setup in which we apply results of the previous two sections.

Let $\Lambda$ be an algebra. Assume we are given full subcategories $\calL$ and $\calR$ of $\mod \Lambda$ such that the following conditions are satisfied:
\begin{enumerate}

\item
$\calL$ and $\calR$ are closed under direct sums and direct summands;

\item
if $M \in \mod \Lambda$, then there exist $U \in \calL$ and $V \in \calR$ with $M \simeq U \oplus V$;

\item
$\pdim_\Lambda U \leq 1$, for each $U \in \calL$, and $\idim_\Lambda V \leq 1$, for each $V \in \calR$;

\item
$\Hom_\Lambda (V, U) = 0 = \Ext_\Lambda^1 (U, V)$, for all $U \in \calL$ and $V \in \calR$;

\item
if $\bd$ is a dimension vectors, then
\begin{gather*}
\calL (\bd) := \{ M \in \mod_\Lambda^\bd (k) : \text{$M \in \calL$} \}
\\
\intertext{and}
\calR (\bd) := \{ M \in \mod_\Lambda^\bd (k) : \text{$M \in \calR$} \}
\end{gather*}
are open and irreducible (if nonempty) subsets of $\mod_\Lambda^\bd (k)$.
\end{enumerate}
We call such a pair of subcategories of $\mod \Lambda$ a geometric bisection.

Observe that in the above situation, if $\bd$ is a dimension vector, then $\mod_\Lambda^\bd (k)$ is the disjoint union of the sets $\calL (\bd_1) \oplus \calR (\bd_2)$, for dimension vectors $\bd_1$ and $\bd_2$ such that $\bd_1 + \bd_2 = \bd$. Here and in the sequel, for subsets $\calZ_1$ and $\calZ_2$ of $\mod_\Lambda^{\bd_1} (k)$ and $\mod_\Lambda^{\bd_2} (k)$, respectively, we put
\begin{multline*}
\calZ_1 \oplus \calZ_2 := \{ M \in \mod_\Lambda^{\bd_1 + \bd_2} (k) :
\\
\text{$M \simeq M_1 \oplus M_2$ for some $M_1 \in \calZ_1$ and $M_2 \in \calZ_2$} \}.
\end{multline*}
We stress that in general $\calZ_1 \oplus \calZ_2$ differs from the set
\begin{multline*}
\calZ_1 \times \calZ_2 := \{ M \in \mod_\Lambda^{\bd_1 + \bd_2} (k) : \text{$M = M_1 \oplus M_2$}
\\
\text{for some $M_1 \in \calZ_1$ and $M_2 \in \calZ_2$} \} \subseteq
\begin{bmatrix}
\mod_\Lambda^{\bd_1} (k) & 0
\\
0 & \mod_\Lambda^{\bd_2} (k)
\end{bmatrix}.
\end{multline*}

In addition to the above we assume the following. There exist $\Lambda$-modules $H_x$, $x \in \bbX$, where $\bbX$ is an index set, such that the following conditions are satisfied:
\begin{enumerate}

\item
if $\bd$ is a dimension vector and
\[
\calU_x (\bd) := \{ M \in \mod_\Lambda^{\bd} (k) : \text{$\Hom_\Lambda (H_x, M) = 0$} \}  \qquad (x \in \bbX),
\]
then
\begin{equation} \label{eq union}
\calL (\bd) = \bigcup_{x \in \bbX} \calU_x (\bd);
\end{equation}

\item
if $\bd$ is a dimension vector and
\[
\calU_x' (\bd) := \{ M \in \mod_\Lambda^{\bd} (k) : \text{$\Ext_\Lambda^1 (H_x, M) = 0$} \}  \qquad (x \in \bbX),
\]
then
\begin{equation} \label{eq intersection}
\calR (\bd) = \bigcap_{x \in \bbX} \calU_x' (\bd).
\end{equation}

\end{enumerate}
In the above situation we say that the geometric bisection $(\calL, \calR)$ is determined by the modules $H_x$, $x \in \bbX$.

The following lemma is a direct consequence of Proposition~\ref{prop main}.

\begin{lemma} \label{lemma sequence}
Let $\bd$ be a dimension vector such that $\calU := \calL (\bd) \neq \emptyset$. If $\calZ := \ol{\calU}$, $N \in \calZ \setminus \calU$, and $\calO_N$ is maximal in $\calZ \setminus \calU$, then there exists an exact sequence
\[
0 \to U \to M \to V \to 0
\]
such that $N \simeq U \oplus V$, $U \in \calL$, $V \in \calR$, and $M \in \calU$.
\end{lemma}

\begin{proof}
There exist $U \in \calL$ and $V \in \calR$ such that $N \simeq U \oplus V$. Without loss of generality we may assume $N = U \oplus V$. Condition~\eqref{eq union} implies that there exists $x \in \bbX$ such that
\[
\min \{ [H_x, M]_\Lambda : \text{$M \in \calZ$} \} = 0.
\]
Put $H := H_x$. Note that $\bdim V \neq 0$, since $N \not \in \calL (\bd)$. Consequently, conditions~\eqref{eq union} and~\eqref{eq intersection} imply
\[
[H, V]_\Lambda > 0 = \min \{ [H, M]_\Lambda : \text{$M \in \calZ$} \}
\]
and $[H, V]_\Lambda^1 = 0$. Now the claim follows from Proposition~\ref{prop main}.
\end{proof}

The following theorem will imply Theorem~\ref{main theorem prim}.

\begin{theorem} \label{theorem main one prim}
Let $\bd$ be a dimension vector such that $\calU := \calL (\bd) \neq \emptyset$. If $N \in \ol{\calU}$, then there exists $M \in \calU$ such that $N \in \ol{\calO}_M$. In other words,
\[
\ol{\calU} = \bigcup_{M \in \calU} \ol{\calO}_M.
\]
\end{theorem}

\begin{proof}
Put $\calZ := \ol{\calU}$. We may assume that $\calO_N$ is maximal in $\calZ \setminus \calU$. By Lemma~\ref{lemma sequence} there exists an exact sequence
\[
0 \to U \to M \to V \to 0
\]
such that $N \simeq U \oplus V$ and $M \in \calU$. Then $\calO_N \subseteq \ol{\calO}_M$, hence the claim follows.
\end{proof}

For the proof of the next result we need the following fact. In the proof of this proposition we only use that $(\calL, \calR)$ is a geometric bisection (not necessarily determined by modules $H_x$, $x \in \bbX$). Consequently, it also has its dual version for $\calU = \calR (\bd)$.

\begin{proposition} \label{prop setV}
Let $\bd$ be a dimension vector such that $\calU := \calL (\bd) \neq \emptyset$. Let $\calZ := \ol{\calU}$, $\calZ'$ be an irreducible component of $\calZ \setminus \calU$, and $\bd_1$ and $\bd_2$ be dimension vectors such that $\calZ' \cap (\calL (\bd_1) \oplus \calR (\bd_2))$ is dense in $\calZ'$. Then there exist $U \in \calL (\bd_1)$, $V \in \calR (\bd_2)$, and an open subset $\calV$ of $\calZ \cap \calE_\Lambda^{\bd_2, \bd_1} (k)$ such that, if $N := U \oplus V$, then $N \in \calZ'$, $\calO_N$ is maximal in $\calZ \setminus \calU$, $N \in \calV$, and $\ext (\calV) = [V, U]_\Lambda^1$.
\end{proposition}

\begin{proof}
We first make an auxiliary observation. Let $N' \simeq U' \oplus V'$, for $U' \in \calL (\bd_1)$ and $V' \in \calR (\bd_2)$. Then
\begin{equation} \label{eq ext2}
[N', N']_\Lambda^2 = [V', U']_\Lambda^2 = \langle \bd_2, \bd_1 \rangle + [V', U']_\Lambda^1.
\end{equation}

For the rest of the proof we fix $U \in \calL (\bd_1)$ and $V \in \calR (\bd_2)$ such that, if $N := U \oplus V$, then $N \in \calZ'$, $N$ does not belong to an irreducible component of $\calZ \setminus \calU$ different from $\calZ'$, $\calO_N$ is maximal in $\calZ \setminus \calU$, and
\[
[N, N]_\Lambda^2 = \min \{ [N', N']_\Lambda^2 : \text{$N' \in \calZ'$} \}.
\]
Put $e := [V, U]_\Lambda^1$. Then~\eqref{eq ext2} implies that
\begin{equation} \label{eq ext2 bound}
[N', N']_\Lambda^2 \geq \langle \bd_2, \bd_1 \rangle + e,
\end{equation}
for each $N' \in \calZ'$.

Let $\calE'$ be the open subset of $\calE_\Lambda^{\bd_2, \bd_1} (k)$ consisting of $\left[
\begin{smallmatrix}
U' & Z
\\
0 & V'
\end{smallmatrix}
\right]$ such that $U' \in \calL (\bd_1)$ and $V' \in \calR (\bd_2)$. Then, $\calZ \cap \calE'$ is an open subset of $\calZ \cap \calE_\Lambda^{\bd_2, \bd_1} (k)$ containing $N$. Let $\calV$ be the subset of $\calZ \cap \calE'$ obtained by subtracting all components of $\calZ \cap \calE_\Lambda^{\bd_2, \bd_1} (k)$, which do not contain $N$. In particular, $N \in \calV$ and $\calV$ is an open subset of $\calZ \cap \calE_\Lambda^{\bd_2, \bd_1} (k)$. We show that $\ext (\calV) = e$, and this will finish the proof.

Assume this is not the case. Let $\calU'$ be the set of $\left[
\begin{smallmatrix}
U' & Z
\\
0 & V'
\end{smallmatrix}
\right]$ in $\calV$ such that $[V', U']_\Lambda^1 < e$. Then there exists an irreducible component $\calV_0$ of $\calV$ such that $\calU_0 := \calV_0 \cap \calU' \neq \emptyset$. Then $\calU_0$ is an open subset of $\calV_0$ and $N \in \calV_0 \subseteq \ol{\calU_0}$, hence there exists an irreducible curve $\calC$ and a regular map $\varphi \colon \calC \to \calV_0$ such that $\varphi (c_0) = N$, for some $c_0 \in \calC$, and $\varphi (\calC) \cap \calU_0 \neq \emptyset$. Write $\varphi = \left[
\begin{smallmatrix}
\varphi_{1, 1} & \varphi_{2, 1}
\\
0 & \varphi_{2, 2}
\end{smallmatrix}
\right]$ (in particular, $\varphi_{1, 1} \colon \calC \to \calL (\bd_1)$ and $\varphi_{2, 2} \colon \calC \to \calR (\bd_2)$) and put
\[
\psi :=
\begin{bmatrix}
\varphi_{1, 1} & 0
\\
0 & \varphi_{2, 2}
\end{bmatrix}
\colon \calC \to \calL (\bd_1) \times \calR (\bd_2).
\]

Obviously $\psi (c_0) = N$. Moreover, $\psi (\calC) \cap \calU_0 \neq \emptyset$. Indeed, fix $c \in \calC$ such that $\varphi (c) = \left[
\begin{smallmatrix}
U' & Z
\\
0 & V'
\end{smallmatrix}
\right] \in \calU_0$. Then $\psi (c) = \left[
\begin{smallmatrix}
U' & 0
\\
0 & V'
\end{smallmatrix}
\right] \in \calU'$. Moreover,
\[
\begin{bmatrix}
U' & \lambda Z
\\
0 & V'
\end{bmatrix} =
\begin{bmatrix}
\lambda \Id & 0
\\
0 & \Id
\end{bmatrix}
\ast
\begin{bmatrix}
U' & Z
\\
0 & V'
\end{bmatrix} \in \calU_0 \subseteq \calV_0,
\]
for each $\lambda \in k \setminus \{ 0 \}$, since $\calU_0$ is invariant under the action of $\GL_{\bd_1} (k) \times \GL_{\bd_2} (k)$. Since $\calV_0$ is closed, $\psi (c) \in \calV_0 \cap \calU' = \calU_0$. Using that $\calZ$ is $\GL_\bd (k)$-invariant and closed, we show similarly that $\psi (\calC) \subseteq \calZ$. Additionally, $\psi (\calC) \cap \calU = \emptyset$, hence $\psi (\calC) \subseteq \calZ \setminus \calU$. Since $\psi (\calC)$ is irreducible, $N \in \psi (\calC)$, and $\calZ'$ is the unique irreducible component of $\calZ \setminus \calU$ containing $N$, $\psi (\calC) \subseteq \calZ'$. Consequently, $\calZ' \cap \calU_0 \neq \emptyset$. However, if $N' \in \calZ \cap \calU_0$, then~\eqref{eq ext2} implies $[N', N']_\Lambda^2 < \langle \bd_2, \bd_1 \rangle_\Lambda + e$, which contradicts~\eqref{eq ext2 bound}.
\end{proof}

The following theorem is a more general version of Theorem~\ref{main theorem}.

\begin{theorem} \label{theorem main two prim}
Assume $\Lambda$ is triangular and $\gldim \Lambda \leq 2$. Let $\bd$ be a dimension vector such that $\calU := \calL (\bd) \neq \emptyset$. Then $\ol{\calU}$ is regular in codimension one.
\end{theorem}

\begin{proof}
Let $\calZ := \ol{\calU}$. Assume first $N \in \calU$. Then $\pdim_\Lambda N \leq 1$. Consequently, we may apply Corollary~\ref{coro nonsingular} with $U = N$ and $V = 0$, and get that $N$ is a nonsingular point of $\calZ$.

Now let $\calZ'$ be an irreducible component of $\calZ \setminus \calU$. In order to finish the proof it is enough to find $N \in \calZ'$ such that $N$ is a nonsingular point of $\calZ$. From Proposition~\ref{prop setV} we know there exist $U \in \calL (\bd_1)$, $V \in \calR (\bd_2)$, and an open subset $\calV$ of $\calZ \cap \calE_\Lambda^{\bd_2, \bd_1} (k)$ such that, if $N := U \oplus V$, then $N \in \calZ'$, $\calO_N$ is maximal in $\calZ \setminus \calU$, $N \in \calV$, and $\ext (\calV) = [V, U]_\Lambda^1$. Moreover, by Lemma~\ref{lemma sequence} there exists an exact sequence
\[
0 \to U \to M \to V \to 0
\]
with $\pdim_\Lambda M \leq 1$, hence the claim follows from Corollary~\ref{coro nonsingular}.
\end{proof}

\section{Applications} \label{sect applications}

In this section we present applications of the results of Section~\ref{section main result}, which include Theorems~\ref{main theorem} and~\ref{main theorem prim}.

\subsection{Periodic modules over concealed canonical algebras}
Let $\Lambda$ be a concealed canonical algebra. Then $\Lambda$ is triangular and $\gldim \Lambda \leq 2$~\cite{LenzingMeltzer}. We describe a geometric bisection of $\mod \Lambda$. According to~\cites{LenzingMeltzer} there exists a sincere separating family $\calT = (\calT_x)_{x \in \bbP^1 (k)}$. Let $\bbX_0$ be the set of $x \in \bbP^1 (k)$, such that $\calT_x$ is homogeneous, and $\bbX \supseteq \bbX_0$ be a set indexing (the isomorphism classes of) the modules lying the mouths of the tubes $\calT_x$, $x \in \bbP^1 (k)$. Fix, for each $x \in \bbX$, the corresponding module $H_x$. Then there exists a dimension vector $\bh$ such that $\bdim H_x = \bh$, for each $x \in \bbX_0$. We put:
\begin{multline*}
\calL := \{ M \in \mod \Lambda : \text{$\langle \bh, X \rangle \leq 0$},
\\
\text{for each indecomposable direct $X$ summand of $M$} \},
\end{multline*}
and
\begin{multline*}
\calR := \{ M \in \mod \Lambda : \text{$\langle \bh, X \rangle > 0$},
\\
\text{for each indecomposable direct $X$ summand of $M$} \}.
\end{multline*}
It follows from the representation theory of concealed canonical algebras (we refer to~\cites{LenzingdelaPena, LenzingSkowronski}) and~\cite{Bobinski2008}*{Section~3}, that $(\calL, \calR)$ is a geometric bisection of $\mod \Lambda$ determined by $H_x$, $x \in \bbX$. Consequently, Theorems~\ref{theorem main one prim} and~\ref{theorem main two prim} give the following.

\begin{theorem} \label{theorem cancealed}
If $\Lambda$ is a concealed canonical algebra and $\bd$ is a dimension vector such that $\calU := \calL (\bd) \neq \emptyset$, then $\ol{\calU}$ is regular in codimension one and
\[
\ol{\calU} = \bigcup_{M \in \calU} \ol{\calO}_M.
\]
\end{theorem}

We explain now how the above implies Theorems~\ref{main theorem} and~\ref{main theorem prim}. If $\Lambda$ is tame, then Theorem~\ref{main theorem} follows from~\cite{BobinskiSkowronski}*{Theorem~1} and there is nothing to prove in Theorem~\ref{main theorem prim}. Thus assume $\Lambda$ is wild. In this case, let $\bd$ is a dimension vector and
\[
\calU := \{ M \in \mod_\Lambda^\bd (k) : \text{$M$ is $\tau$-periodic} \}.
\]
If $\calU \neq \emptyset$, then $\calU = \calL (\bd)$ (by the representation theory of concealed canonical algebras), thus Theorems~\ref{main theorem} and~\ref{main theorem prim} are direct consequences of Theorem~\ref{theorem cancealed}.

We make one more comment. Let $\Lambda$ be a wild concealed canonical algebra and $\calU$ as above. Assume $\calU \neq \emptyset$ and put $\calZ := \ol{\calU}$. Theorem~\ref{main theorem prim} implies that if $\calO_M$ is maximal in $\calZ$, then $M$ is $\tau$-periodic. The representation theory of concealed canonical algebras implies that $M$ is a direct sum of modules from $\calT$. Then it is standard that we may rewrite the description of maximal orbits in $\calZ$ from~\cite{BobinskiSkowronski}*{Proposition~5} (see also~\cite{Ringel1980}*{Theorem~3.5}).

\subsection{Directing modules} \label{sub directing}
A $\Lambda$-module $M$ is called directing if there is no sequence
\[
X_0 \to X_1 \to \cdots \to X_n
\]
of nonzero maps between indecomposable modules such that $X_0$ and $X_n$ are direct summands of $M$, and there exists $0 < i < n$, such that $\tau X_{i + 1} \simeq X_{i - 1}$. The following is the main result of~\cite{Bobinski2009}.

\begin{theorem} \label{theorem direct}
If $M$ is a directing, then $\ol{\calO}_M$ is regular in codimension one.
\end{theorem}

The proof of Theorem~\ref{theorem direct} in~\cite{Bobinski2009} contains a gap. Roughly speaking, we consider the intersection of $\ol{\calO}_M$ with $\mod_\Lambda^{\bd_1} \times \mod_\Lambda^{\bd_2}$ (for suitable dimension vectors $\bd_1$ and $\bd_2$) and assume, without proving, it is a reduced scheme. Now we explain how to correct the proof.

The detailed analysis presented in~\cite{Bobinski2009} shows that we may assume $\Lambda$ is triangular, $\gldim \Lambda \leq 2$, we have a geometric bisection $(\calL, \calR)$ of $\mod \Lambda$, and $\ol{\calO}_M = \ol{\calR (\bdim M)}$. Moreover, \cite{Bobinski2009}*{Corollary~4.4} implies that if $\calO_N$ is maximal in $\ol{\calO}_M \setminus \calO_M$, then there exists an exact sequence
\[
0 \to U \to M \to V \to 0
\]
such that $N \simeq U \oplus V$, $\pdim_\Lambda M \leq 1$ and $\idim_\Lambda V \leq 1$. Consequently, Theorem~\ref{theorem direct} follows from the dual of Proposition~\ref{prop setV} and Corollary~\ref{coro nonsingular}.

\subsection{Irreducible components of module varieties for dimension vectors of indecomposable modules over tame quasi-titled algebras}
An algebra $\Lambda$ is called quasi-tilted if $\gldim \Lambda \leq 2$ and, for each indecomposable $\Lambda$-module $X$, $\pdim_\Lambda X \leq 1$ or $\idim_\Lambda X \leq 1$. The following theorem is the main result of~\cite{Bobinski2012}.

\begin{theorem}
Let $\Lambda$ be a tame quasi-tilted algebra and $\bd$ the dimension vector of an indecomposable $\Lambda$-module. If $\calZ$ is an irreducible component of $\mod_\Lambda^\bd (k)$, then $\calZ$ is nonsingular in codimension one.
\end{theorem}

The proof of this theorem in~\cite{Bobinski2012} contains a gap similar to that described in Subsection~\ref{sub directing}. However, in the situation in which the gap occurs we have a geometric bisection $(\calL, \calR)$ determined by a family of modules such that $\calZ = \ol{\calL (\bd)}$. Since every quasi-tilted algebra is triangular~\cite{HappelReitenSmalo}*{Propostion~III.1.1}, we may use Theorem~\ref{theorem main two prim} instead of the original arguments.

\bibsection

\end{document}